\documentclass[11pt]{amsart}
\usepackage[latin1]{inputenc}

\usepackage{amsmath, amsfonts, amssymb, amsthm}
\usepackage{enumerate}
\usepackage{nicefrac}
\usepackage[numbers]{natbib}
\usepackage[pagebackref,colorlinks,linkcolor=red,citecolor=blue,urlcolor=blue,hypertexnames=true]{hyperref}

\usepackage{tikz}
\usetikzlibrary{matrix,arrows}
\usepackage[onehalfspacing]{setspace}

\newtheorem{theorem}{Theorem}[section]
\newtheorem*{thm57}{Theorem 5.7}
\newtheorem{lemma}[theorem]{Lemma}
\newtheorem{corollary}[theorem]{Corollary}

\theoremstyle{definition}
\newtheorem{definition}[theorem]{Definition}

\newtheorem{remark}[theorem]{Remark}
\newcommand{\N}{\mathbb{N}}
\newcommand{\Z}{\mathbb{Z}}

\newcommand{\R}{\mathbb{R}}
\newcommand{\C}{\mathbb{C}}
\newcommand{\K}{\mathbb{K}}

\newcommand{\Heven}[2]{H^{\text{even}}(#1, #2)}

\newcommand{\wh}{\widehat}
\newcommand{\mcal}{\mathcal}

\newcommand{\Lspin}{L_{\text{spin}}}

\newcommand{\Endo}[1]{\text{\rm End}(#1)}
\newcommand{\Homom}[2]{\text{Hom}(#1,#2)}
\newcommand{\bgend}[1]{\text{end}(#1)}

\newcommand{\trivial}[1]{\underline{#1}}
\newcommand{\id}{\text{id}}

\newcommand{\para}[1]{\mathcal{P}_{#1}}
\newcommand{\pEndo}[2]{\mathcal{P}(#1,#2)}
\newcommand{\hol}[2]{{\rm hol}(#1,#2)}
\newcommand{\Cmax}[2]{C^*_{\rm max}(#1 \to #2)}

\newcommand{\wtM}{\widetilde{M}}
\newcommand{\Vmax}{\mathcal{V}_{\rm max}}
\newcommand{\Vred}{\mathcal{V}_{\rm red}}

\DeclareMathOperator{\ch}{{\rm ch}}

\title[Twisted $K$-theory and obstructions against psc metrics]{Twisted $K$-theory and \\ obstructions against positive scalar curvature metrics}
\author{Ulrich Pennig}
\date{\today}
\address{Mathematisches Institut,\ \ Westf\"alische Wilhelms-Universit\"at M\"unster,\ \ Einstein\-stra\ss e~62, 48149 M\"unster, Germany, \textit{\small E-mail: u.pennig@uni-muenster.de}}

\begin{document}
\begin{abstract}
We decompose the twisted index obstruction $\theta(M)$ against positive scalar curvature metrics for oriented manifolds with spin universal cover into a pairing of a twisted $K$-homology with a twisted $K$-theory class and prove that $\theta(M)$ does not vanish if $M$ is an orientable enlargeable manifold with spin universal cover.
\end{abstract}

\maketitle

\section{Introduction}
In \cite{paper:RosenbergPositiveIII} Rosenberg constructed an index obstruction $\alpha(M) \in KO_n(C^*_{\R}(\pi_1(M)))$ for closed spin manifolds $M$ of dimension $n$, which vanishes if $M$ admits a metric of positive scalar curvature. It takes values in the $K$-theory of the (maximal or reduced) real group $C^*$-algebra associated to the fundamental group and relies on the existence of a spin structure. Gromov, Lawson and Rosenberg conjectured that $\alpha(M)$ is the only obstruction to the existence of a psc metric if  $\dim(M) \geq 5$. This was proven to be true in the simply-connected case by Stolz \cite{paper:StolzSimplyConn}, but is false in general as was shown by Schick \cite{paper:SchickCounter}.

Stolz generalized $\alpha(M)$ to the case, where $M$ itself may not be spin, but its universal cover $\wtM$ still is \cite{paper:StolzConcordance}, \cite[Theorem 1.7]{paper:RosenbergProgressRep}. The new invariant $\theta(M) \in KO(C^* \gamma)$ takes values in the $K$-theory of a real $C^*$-algebra associated to a \emph{twisted} version of the fundamental group accounting for the missing spin structure on $M$. The latter is a $\Z/2\Z$-extension $\wh{\pi}$ of $\pi=\pi_1(M)$, therefore we will use the notation $C^*(\wh{\pi} \to \pi)$ instead of $C^*\gamma$.

The element $\alpha(M)$ can be expressed as the pairing of the Dirac class $[D] \in KO_n(M)$ with the $KO$-theory class $[\mcal{V}] \in KO_0(C(M, C^*\pi))$ of the Mishchenko-Fomenko bundle. We show that the same is true for $\theta(M)$ if one switches to the \emph{twisted} versions of $K$-homology and $K$-theory \cite{paper:DonovanKaroubi, paper:AtiyahSegal1, paper:KaroubiOldAndNew}. For simplicity we treat the complex analogues of the real invariants and refer to \cite[Section 1.4]{paper:HankeSchick} for a nice exposition on how the different indices are related. We identify $\theta(M)$ as the pairing of the twisted fundamental $K$-homology class $[D^S] \in KK(C(M,\C\ell(M)), \C)$ with a twisted version of the Mishchenko-Fomenko bundle representing an element $[\mathcal{V}^S] \in KK(\C, C(M,\C\ell(M)) \otimes C^*(\wh{\pi} \to \pi))$:
\[
	\theta(M) = {\rm ind}(D_+^{\mathcal{V}}) = [\mathcal{V}^{S}] \otimes_{C(M,\C\ell(M))} [D^{S}]\ .
\]
The Dixmier-Douady class of the twist in this case is the element $W_3(M) := \beta(w_2(M)) \in H^3(M; \Z)$, i.e.\ the Bockstein of the second Stiefel-Whitney class. The superscript $S$ denotes the twisted spinor bundle over $M$. We will use the language of bundle gerbes developed by Murray \cite{paper:BundleGerbes} and the $C^*$-algebraic versions of their modules \cite{paper:KTheoryBGM} to obtain a geometric description of the twisted $K$-group with coefficients in a $C^*$-algebra. Most features of index theory are preserved in the twisted case: Murray and Singer proved the analogue of the Atiyah-Singer index theorem in this context \cite{paper:GerbesAndIndexThms} and Carey and Wang proved the Thom isomorphism~\cite{paper:Wang} (see also \cite{paper:DonovanKaroubi}). 

As an application of our techniques, we show that $\theta(M)$ does not vanish for orientable enlargeable manifolds (Definition \ref{def:enlargeable}) with spin universal cover and thereby enhance a result of Hanke and Schick \cite{paper:HankeSchick, paper:HankeSchickInfinite}: 
\begin{thm57}
Let $M$ be a closed compact smooth orientable even-dimensional manifold with $\dim(M) \geq 3$ and $\wtM$ spin that is enlargeable in the sense of Definition~\ref{def:enlargeable}. Then we have
\[
	\theta^{\rm max}(M) \neq 0 \ .
\]
\end{thm57}
Gromov and Lawson showed that enlargeable manifolds $M$ do not allow a metric of positive scalar curvature. They worked with finite covers in the definition of enlargeability in \cite{paper:GromovLawsonFund}, but later generalized to infinite ones \cite{paper:GromovLawsonIHES83}. We also allow the covers $\bar{M} \to M$ to be \emph{non-compact} as in \cite{paper:HankeSchickInfinite}. In this generality, transfer arguments fail, but twisted Hilbert $A$-module bundles, provide a way to circumvent this. This result is independent of the injectivity of the twisted Baum-Connes map. \\
\paragraph{\bf Acknowledgments}
The paper is an excerpt of the author's PhD thesis under supervision of Prof. Thomas Schick to whom the author would like to express his deep gratitude. Apart from that, he would like to thank Bernhard Hanke and Andreas Thom for many stimulating discussions, Max Karoubi for pointing out valuable references about twisted K-theory and the German National Academic Foundation for financial support.\\

\paragraph{\bf Notation}
Throughout the paper, $M$ will denote a smooth compact closed orientable manifold and $A$ will be a unital $C^*$-algebra if not stated otherwise. Moreover, $\underline{F}$ will denote a trivial bundle with fiber $F$ if the base space is clear. Whenever we have a surjective submersion $Y \to M$, the notation $Y^{[n]}$ will denote the $n$-fold fiber product over $M$.

\section{Twisted Hilbert $A$-module bundles}
In this section we will discuss a way to represent classes in twisted $K$-theory with coefficients in a $C^*$-algebra $A$ by twisted Hilbert $A$-module bundles. These are straightforward generalizations of bundle gerbe modules \cite{paper:KTheoryBGM}. For an intro\-duction to Hilbert $A$-modules we refer the reader to \cite{book:Lance}. We start by reviewing the geometric description of $2$-cocycles via bundle gerbes \cite{paper:BundleGerbes, paper:IntroToBGs}. 

\begin{definition}
Let $M$ be a smooth manifold and let $Y \to M$ be a surjective submersion. A \emph{real} line bundle $L \to Y^{[2]}$ will be called a \emph{$\Z/2\Z$-bundle gerbe} (or simply \emph{bundle gerbe} for short) if there exists a multiplication over $Y^{[3]}$, i.e.\ an isomorphism of line bundles
\(
	\mu \colon \pi_{12}^* L \otimes \pi_{23}^* L \to \pi_{13}^* L
\),
where $\pi_{ij} \colon Y^{[3]} \to Y^{[2]}$ denotes the canonical projections to the
fiber product of the $i$th and $j$th factor, and such that over $Y^{[4]}$ the following
diagram commutes:
\begin{center}
\begin{tikzpicture}
\matrix (m) [matrix of math nodes, row sep=1cm, column sep=0.3cm, 
             text height=1.5ex, text depth = 0.25ex]
{ (\pi_{12}^* L \otimes \pi_{23}^* L) \otimes \pi_{34}^* L & & \pi_{12}^* L \otimes (\pi_{23}^* L \otimes \pi_{34}^* L) \\
\pi_{13}^* L \otimes \pi_{34}^* L & \pi_{14}^* L & \pi_{12}^* L \otimes \pi_{24}^* L \\};
\path[->,font=\scriptsize]
(m-1-1) edge node[left] {$\mu \otimes \text{id}$} (m-2-1)
(m-1-3) edge node[auto] {$\text{id} \otimes \mu$} (m-2-3)
(m-2-1) edge node[below] {$\mu$} (m-2-2)
(m-2-3) edge node[auto] {$\mu$} (m-2-2);
\path[font=\normalfont] (m-1-1.base east) edge (m-1-3.base west)
                        (m-1-1.east) edge (m-1-3.west);
\end{tikzpicture}
\end{center}
The bundle gerbe $\delta Q = \pi_1^*Q \otimes \pi_2^*Q^* \to Y^{[2]}$ associated to a line bundle $Q \to Y$ will be called the \emph{trivial bundle gerbe}. Let $L \to Y^{[2]}$ be a bundle gerbe. A choice of a line bundle $Q \to Y$ together with an isomorphism of bundle gerbes $L \to \delta Q$ will be called a \emph{trivialization} of $L$.
\end{definition}

Let $L_i \to Y_i^{[2]}$ for $i \in \{1,2\}$ be two bundle gerbes and let $\pi_i \colon Y_1^{[2]} \times_M Y_2^{[2]} \to Y_i^{[2]}$ denote the projection. The exterior tensor product $L_1 \boxtimes L_2 = \pi_1^*L_1 \otimes \pi_2^*L_2$ is again a bundle gerbe. Each  bundle gerbe $L$ has a class $dd(L) \in H^2(M, \Z/2\Z)$ canonically associated to it as explained for the case of $S^1$-bundle gerbes in \cite{paper:BundleGerbes}. $dd(L)$ is called the \emph{Dixmier-Douady class} of $L$ and we summarize its properties in the following theorem proven in \cite{paper:BundleGerbes}:
\begin{theorem}
Let $L \to Y^{[2]}$ and $L_i \to Y_i^{[2]}$ for $i \in \{1,2\}$ be bundle gerbes. The Dixmier-Douady class has the following properties:
\begin{enumerate}[a)]
\item \label{it:ddLzero} $dd(L) = 0$ if and only if $L$ is isomorphic over $Y^{[2]}$ to a trivial bundle gerbe. 
\item $dd(L_1 \boxtimes L_2) = dd(L_1) + dd(L_2)$. \hfill $\square$ 
\end{enumerate}
\end{theorem}

\begin{definition}\label{def:bg_conn}
Let $L$ be a bundle gerbe. A covariant derivative $\nabla^L \colon \Omega^0(L) \to \Omega^1(L)$ on $L$ is called a \emph{bundle gerbe connection} if the multiplication isomorphism
\(
	\mu \colon \pi_{12}^* L \otimes \pi_{23}^* L \to \pi_{13}^* L
\)
pulls it back to the canonical connection on the tensor product, i.e.
\begin{equation} \label{eqn:covder}
	\mu^* \pi_{13}^* \nabla^L = \pi_{12}^* \nabla^L \otimes 1 + 1 \otimes \pi_{23}^* \nabla^L\ .
\end{equation}
\end{definition}

\begin{remark}\label{rem:everything_flat}
The proof for the existence of such connections given in \cite{paper:BundleGerbes} works with the obvious changes for $\Z/2\Z$-bundle gerbes as well. Since the structure group is discrete in our case, every $\Z/2\Z$-bundle gerbe connection is automatically flat.
\end{remark}

The main example of bundle gerbes will arise from the following construction.
\begin{definition} \label{def:lifting_bg}
Let $\Gamma$ be a Lie group (possibly discrete) and let $q \colon \widehat{\Gamma} \to \Gamma$ be a central $\Z/2\Z$-extension of $\Gamma$. Let $P \to M$ be a principal $\Gamma$-bundle over a manifold $M$. The line bundle $L_P$ associated to the principal $\Z/2\Z$-bundle $\widehat{L}_P \to P^{[2]}$ with
\(
	\widehat{L}_P = \{ (p_1,p_2,\widehat{g}) \in P^{[2]} \times \widehat{\Gamma}\ |\ p_1\, q(\widehat{g}) = p_2 \}
\)
is a bundle gerbe, which will be called the \emph{lifting bundle gerbe} associated to the extension. 
\end{definition}
For details about this construction see \cite[section 6.1]{paper:IntroToBGs}. Since the principal $\Z/2\Z$-bundle associated to any trivialization $Q$ of $L_P$ is a lift of $P$ to a principal $\wh{\Gamma}$-bundle, we have

\begin{lemma}[\cite{paper:BundleGerbes}] \label{lem:lifting} 
The class $dd(L_P)$ represents the obstruction of lifting $P \to M$ to a principal $\widehat{\Gamma}$-bundle. The isomorphism classes of trivialization of $L_P$ are in $1:1$-correspondence with the possible lifts of $P$ to a principal $\wh{\Gamma}$-bundle. 
\end{lemma}

\begin{definition} \label{def:TwistedHilbertA}
Let $M$ be a smooth manifold, $A$ a $C^*$-algebra and let $Y \to M$ be a surjective submersion. Let $L \to Y^{[2]}$ be a bundle gerbe. A (right) Hilbert $A$-module bundle $E \to Y$ together with an $A$-linear action 
\(
\gamma \colon L \otimes \pi_2^*E \to \pi_1^*E
\)
is called a \emph{twisted Hilbert $A$-module} bundle for $L$ if
the following associativity diagram commutes:
\begin{center}
\begin{tikzpicture}
\matrix (m) [matrix of math nodes, row sep=1cm, column sep=0.3cm, 
             text height=1.5ex, text depth = 0.25ex]
{ (\pi_{12}^* L \otimes \pi_{23}^* L) \otimes \pi_{3}^* E & & \pi_{12}^* L \otimes (\pi_{23}^* L \otimes \pi_{3}^* E) \\
\pi_{13}^* L \otimes \pi_{3}^* E & \pi_{1}^* E & \pi_{12}^* L \otimes \pi_{2}^* E \\};
\path[->,font=\scriptsize]
(m-1-1) edge node[left] {$\mu \otimes \text{id}$} (m-2-1)
(m-1-3) edge node[auto] {$\text{id} \otimes \gamma$} (m-2-3)
(m-2-1) edge node[below] {$\gamma$} (m-2-2)
(m-2-3) edge node[auto] {$\gamma$} (m-2-2);
\path[font=\normalfont] (m-1-1.base east) edge (m-1-3.base west)
                        (m-1-1.east) edge (m-1-3.west);
\end{tikzpicture}
\end{center}
$E$ will be called \emph{finitely generated} and \emph{projective}, if its fibers are finitely generated and projective as Hilbert $A$-modules. $\gamma$ will be called the \emph{twisting} of $E$.

Let $E$, $E'$ be two twisted Hilbert $A$-module bundles for the same bundle gerbe $L$ and denote the twistings by $\gamma$ and $\gamma'$. A right $A$-linear map $f \colon E \to E'$ will be called a \emph{morphism of twisted Hilbert $A$-module bundles} or (\emph{twisted morphism} for short) if the following diagram commutes: 
\begin{center}
\begin{tikzpicture}
\matrix (m) [matrix of math nodes, row sep=1.3cm, column sep=1.3cm, 
             text height=1.5ex, text depth = 0.25ex]
{ L \otimes \pi_2^* E & \pi_1^* E \\
  L \otimes \pi_2^* E' & \pi_1^* E' \\
};
\path[->,font=\scriptsize]
(m-1-1) edge node[auto] {$\gamma$} (m-1-2)
(m-2-1) edge node[auto] {$\gamma'$} (m-2-2)
(m-1-1) edge node[auto] {$\text{id}_L \otimes \pi_2^*f$} (m-2-1)
(m-1-2) edge node[auto] {$\pi_1^*f$} (m-2-2);
\end{tikzpicture}
\end{center}
\end{definition}
For $A = \C$ a twisted Hilbert $A$-module bundle is the same as a bundle gerbe module \cite{paper:KTheoryBGM}. Let $L_i \to Y_i^{[2]}$ for $i \in \{1,2\}$ be two bundle gerbes with $dd(L_1) = dd(L_2)$. A trivialization of $L_1^* \boxtimes L_2$ should be seen as a generalized morphism between $L_1$ and $L_2$ \cite{paper:Waldorf}. In particular, it induces a push-forward from twisted bundles for $L_1$ to twisted bundles for $L_2$.

\begin{lemma}\label{lem:descentlemma}
Let $E \to Y$ be a locally trivial bundle of right Hilbert $A$-modules with fiber $V$ over the domain space of a surjective submersion $\pi \colon Y \to M$. Assume that there is a bundle isomorphism
\[
	\phi \colon \pi_2^* E \overset{\cong}{\longrightarrow} \pi_1^*E \ ,
\]
where $\pi_i \colon Y^{[2]} \to Y$ denote the canonical projections. If the corresponding associativity diagram over $Y^{[3]}$ commutes, then there is a Hilbert $A$-module bundle $\widetilde{E} \to M$ with fiber $V$ and an isomorphism $E \to \pi^*\widetilde{E}$. $\widetilde{E}$ is unique up to isomorphism.
\end{lemma}
\begin{proof}
Cover $M$ by open sets $U_i$ such that there exist sections $\sigma_i \colon U_i \to Y$. Let $E_i = \sigma_i^* E$ and denote the maps induced by $\phi$ on the double intersections $U_{ij} = U_i \cap U_j$ by $\phi_{ij} \colon E_j \to E_i$. Now we can define
\[
\widetilde{E} = \coprod_{i \in I} E_i \ /\ \sim
\]
where the equivalence relation is induced by the maps $\phi_{ij}$.
\end{proof}

\begin{lemma} \label{lem:transfer}
Let $L_i \to Y_i^{[2]}$ for $i \in \{1,2\}$ be two bundle gerbes with $dd(L_1) = dd(L_2)$. Let $Q \to Y_1 \times_M Y_2$ be a trivialization of $L_1^* \boxtimes L_2 \to (Y_1 \times_M Y_2)^{[2]}$. If $E \to Y_1$ is a twisted bundle for $L_1$ and $\rho \colon Y_1 \times_M Y_2 \to Y_1$ denotes the canonical projection, then $\rho^*E \otimes Q$ descends to a twisted bundle $Q(E) \to Y_2$  for $L_2$. In particular if $L_2 = \trivial{\C}$ over $M^{[2]} = M$, i.e.\ $Q$ is a trivialization of $L_1^*$, we obtain an (untwisted) Hilbert $A$-module bundle $Q(E) \to M$.
\end{lemma}

\begin{proof}
In view of the last lemma we have to prove that $\pi_2^*(\rho^*E \otimes Q) \cong \pi_1^*(\rho^*E \otimes Q)$, where $\pi_i \colon Y_1^{[2]} \times_M Y_2 \to Y_1 \times_M Y_2$ is the projection to the $i$th $Y_1$-factor. Moreover, we need to check that associativity holds and that the result carries an $L_2$-twisting. Since $Q$ trivializes $L_1^* \boxtimes L_2$ we have the following isomorphism over $(y^i_1, y^i_2) \in Y_i^{[2]}$:
\[
(L_1^*)_{(y^1_1, y^1_2)} \otimes (L_2)_{(y^2_1, y^2_2)} \to Q_{(y^1_1, y^2_1)} \otimes (Q^*)_{(y^1_2, y^2_2)}\ .
\]
The fiber $(L_i)_{(y,y)}$ can be canonically identified with $\C$ in a way which is compatible with the product operation. In particular, we get $(L_1^*)_{(y^1_1, y^1_2)} \otimes Q_{(y^1_2, y^2)} \to Q_{(y^1_1, y^2)}$ by setting $y^2_1 = y^2_2 = y^2 \in Y_2$. From this, we obtain
\[
E_{y^1_2} \otimes Q_{(y^1_2, y^2)} \to (L_1)_{(y^1_1, y^1_2)} \otimes E_{y^1_2} \otimes (L_1^*)_{(y^1_1, y^1_2)} \otimes Q_{(y^1_2, y^2)} \to E_{y^1_1} \otimes Q_{(y^1_1, y^2)}\ .
\]
The first map is the canonical identification $L_1 \otimes L_1^* \to \trivial{\C}$, the second is induced by the action of $L_1$ on $E$ and the above map. Likewise, the other diagonal embedding $Y_1 \times_M Y_2^{[2]} \to (Y_1 \times_M Y_2)^{[2]}$ yields an isomorphism
\(
(L_2)_{(y^2_1, y^2_2)} \otimes Q_{(y^1, y^2_2)} \to Q_{(y^1, y^2_1)} 
\)
for $(y^2_1, y^2_2) \in Y_2^{[2]}$ and $y^1 \in Y_1$. After tensoring with $E$ this becomes the twisting map
\(
(L_2)_{(y^2_1, y^2_2)} \otimes E_{y^1} \otimes Q_{(y^1, y^2_2)} \to E_{y^1} \otimes Q_{(y^1, y^2_1)}\ .
\)

By the associativity of $L_1^* \boxtimes L_2$ and the fact that the above two maps are derived from a bundle gerbe isomorphism, this twisting commutes with the action that is used to define the descended bundle $Q(E)$. Moreover, this implies the associativity of the twisting as well as the associativity of the descend isomorphism. 
\end{proof}

\begin{definition}\label{def:twisted_K}
Let $L \to Y^{[2]}$ be a bundle gerbe over a surjective submersion $Y \to M$ and let $A$ be a unital $C^*$-algebra. Denote by $K^0_{L,A}(M)$ the Grothendieck group of isomorphism classes of finitely generated projective twisted Hilbert $A$-module bundles for $L$.
\end{definition}

Given a principal $PO(n)$-bundle $P \to M$ there is a bundle of matrix algebras $\mathcal{K} \to M$ with fiber $M_n(\C)$ associated to it via the inclusion $PO(n) \to PU(n)$ and the adjoint action of $PU(n)$. Let $L_P \to P^{[2]}$ be the corresponding lifting bundle gerbe for the extension $O(n) \to PO(n)$. There is a \emph{canonical twisted vector bundle} $S \to P$ for $L_P$ associated to this construction as follows: Let $S = P \times \C^n$ with the twisting given by
\[
	\gamma \colon L_P \otimes \pi_2^*S \to \pi_1^*S \quad ; \quad \gamma([\wh{g}, \lambda] \otimes v) = \lambda\,\wh{g}\,v
\]
where $\widehat{g} \in O(n)$ is a lift of the group element $g \in PO(n)$ determined by the underlying points $(p_1,p_2) \in P^{[2]}$, $\lambda \in \R$ and $(\widehat{g}, -\lambda) \sim (-\widehat{g}, \lambda)$. Given a twisted Hilbert $A$-module bundle $E \to Y$ for some bundle gerbe $L$, note that $S^* \boxtimes E := \pi_P^*S^* \otimes \pi_Y^*E$ over $P \times_M Y$ is a twisted Hilbert $A$-module bundle for $L_P^* \boxtimes L$.

\begin{theorem}\label{thm:Twisted_K}
Let $L$ be a bundle gerbe, let $A$ be a unital $C^*$-algebra. Let $P$, $L_P$ and $\mathcal{K}$ be as in the above paragraph. Assume that $dd(L) = dd(L_P)$ and let $Q$ be a trivialization of $L_P^* \boxtimes L$, then 
\[
	\kappa_Q \colon K^0_{L,A}(M) \to K_0(C(M,\mathcal{K}) \otimes A) \quad ; \quad [E] \mapsto [C(M,Q(S^* \boxtimes E))]
\]
is a well-defined isomorphism.
\end{theorem}

\begin{proof}\label{pf:Twisted_K}
Let $V$ be the typical fiber of $E$. It is a finitely generated projective right Hilbert $A$-module, therefore $(\C^n)^* \otimes V$ is a finitely generated projective right Hilbert $M_n(\C) \otimes A$-module in a canonical way. Thus, $S^* \boxtimes E$ is a bundle of finitely generated projective right Hilbert $M_n(\C) \otimes A$-modules over $P \times_M Y$. When descending this bundle as sketched in Lemma \ref{lem:descentlemma}, we first obtain bundles $(\C^n)^* \otimes E_i \to U_i$ over an open cover $U_i \subset M$. The transitions over double intersections act via the canonical action of $O(n) \subset U(n)$ on the first tensor factor and in an $A$-linear way on the second tensor factor. Therefore, $C(M, Q(S^* \boxtimes E))$ carries an action of $C(M,\mathcal{K} \otimes A)$ and the fiberwise $M_n(\C) \otimes A$-valued inner product on $S^* \boxtimes E$ turns $C(M, Q(S^* \boxtimes E))$ into a right Hilbert $C(M,\mathcal{K} \otimes A)$-module. Since $V$ was finitely generated and projective, the same holds true for $C(M, Q(S^* \boxtimes E))$.

Let $\sigma$ be a projection-valued section of the bundle $\mathcal{K} \otimes M_k(A) \to M$. Let $\pi \colon P \to M$ be the bundle projection, then $\sigma \circ \pi$ yields a section of $\pi^*\mathcal{K} \otimes M_k(A)$. This bundle has a canonical trivialization, such that $\sigma \circ \pi$ yields a continuous map $\widehat{\sigma} \colon P \to M_n(\C) \otimes M_k(A)$ with the property $\widehat{\sigma}(pg) = \widehat{g}^{-1}\,\sigma(p)\,\widehat{g}$. Let $E_{\sigma} = \{ (p,v) \in P \times (\C^n \otimes A^{k})\ |\ \widehat{\sigma}(p)\,v = v\}$. This is a finitely generated, projective twisted Hilbert $A$-module bundle for $L_P$. The map $[\sigma] \mapsto [E_{\sigma}]$ is a homomorphism $K_0(C(M,\mathcal{K}) \otimes A) \to K^0_{L_P,A}(M)$. The composition with the isomorphism $K^0_{L_P,A}(M) \to K^0_{L,A}(M)$ induced by $Q^*$ yields the inverse of the above construction.
\end{proof}

\begin{lemma}\label{lem:comparison}
Consider principal $PO(n_i)$-bundles $P_i$ for $i \in \{1,2\}$ and let $\mathcal{K}_i$ be the associated bundles of complex matrix algebras. We have $dd(L_{P_1}) = dd(L_{P_2})$ if and only if $\mathcal{K}_1 \otimes \mathcal{K}_2 \cong \Endo{\mathcal{V}} \otimes \C$ for a real vector bundle $\mathcal{V} \to M$. The isomorphism classes of trivializations of $L_{P_1}^* \boxtimes L_{P_2}$ are in $1:1$-correspondence with the isomorphism classes of possible $\mathcal{V}$'s.
\end{lemma}

\begin{proof} \label{pf:comparison}
Let $O(n_1) \otimes O(n_2)$ be the quotient of the product $O(n_1) \times O(n_2)$ by the diagonal $Z/2\Z$-action. If $dd(L_{P_1}) = dd(L_{P_2})$, then $P_1 \times_M P_2$ lifts to a principal $O(n_1) \otimes O(n_2)$-bundle $P$ and $\mathcal{V} = P \times_{\rho} (\R^{n_1} \otimes \R^{n_2})$ satisfies the condition, where $\rho$ is the standard representation. The isomorphism class of $\mathcal{V}$ only depends on the one of $P$ and the choices of possible $P$'s are in $1:1$-correspondence with the trivializations of $L_{P_1}^* \boxtimes L_{P_2}$. 

On the other hand, if $\mathcal{V}$ exists, let $P$ be its principal $O(n_1n_2)$-bundle. $P/(\Z/2\Z)$ reduces to the principal $PO(n_1) \times PO(n_2)$-bundle $P_1 \times_M P_2$ and the pullback of $P \to P/(\Z/2\Z)$ via $P_1 \times_M P_2 \to P/(\Z/2\Z)$ is a lift of $P_1 \times_M P_2$ to a principal $O(n_1) \otimes O(n_2)$-bundle, which shows that $dd(L_{P_1}) = dd(L_{P_2})$. 
\end{proof}

\begin{remark} \label{rem:compatibility}
Note that $C(M, \mathcal{V} \otimes \C) \otimes A$ provides a Morita equivalence bimodule between $C(M, \mathcal{K}_1) \otimes A$ and $C(M, \mathcal{K}_2) \otimes A$ and $\psi_{\mathcal{V}} \colon K_0(C(M, \mathcal{K}_1) \otimes A) \to K_0(C(M, \mathcal{K}_2) \otimes A)$ with $\psi_{\mathcal{V}}([X]) = [X \otimes_{C(M,\mathcal{K}_1) \otimes A} C(M, \mathcal{V} \otimes \C) \otimes A]$ is the corresponding isomorphism on $K$-theory. 

Let $L$ be a bundle gerbe. Let $\mathcal{K}_i$, $P_i$ for $i\in \{1,2\}$ be as in Lemma \ref{lem:comparison}. Suppose that $dd(L) = dd(L_{P_1}) = dd(L_{P_2})$. Trivializations $Q_i$ of $L^* \boxtimes L_{P_i}$ induce a trivialization of $L_{P_1}^* \boxtimes L_{P_2}$, since $Q_1^* \boxtimes Q_2$ over $P_1 \times_M P \times_M P \times_M P_2$ descends to a line bundle $Q_{12}$ over $P_1 \times_M P_2$ in a way compatible with the actions of $L_{P_i}$. The trivialization $Q_{12}$ corresponds to a real vector bundle $\mathcal{V}_{12}$ as described in Lem\-ma~\ref{lem:comparison}. Let $\kappa_{Q_i}([E]) = [C(M, Q_i(S_i^* \boxtimes E))]$ be the homomorphism from Theorem~\ref{thm:Twisted_K} associated to the data $Q_i$ and the canonical bundle $S_i$. A bundle gerbe $L$ provides a canonical trivialization of $L^* \boxtimes L$ (in particular, if $P$ is a principal $PO(n)$-bundle, then the principal $PO(n) \times PO(n)$-bundle $P^{[2]}$ has a canonical lift to a $O(n) \otimes O(n)$-bundle). Let $\kappa_{L_{P_i}}$ be the associated isomorphism from Theorem \ref{thm:Twisted_K}. Then the the following diagram of isomorphisms commutes:
\begin{center}
\begin{tikzpicture}
\matrix (m) [matrix of math nodes, row sep=0.7cm, column sep=1.3cm, 
             text height=1.5ex, text depth = 0.25ex]
{ & K_0(C(M, \mathcal{K}_1) \otimes A) & K^0_{L_{P_1}, A}(M) \\
K^0_{L,A}(M) \\
& K_0(C(M, \mathcal{K}_2) \otimes A) & K^0_{L_{P_2}, A}(M) \\
};
\path[->,font=\scriptsize]
(m-2-1) edge node[above] {$\kappa_{Q_1}$} (m-1-2)
(m-2-1) edge node[below] {$\kappa_{Q_2}$} (m-3-2)
(m-1-2) edge node[auto] {$\psi_{\mathcal{V}}$} (m-3-2)
(m-1-3) edge node[above] {$\kappa_{L_{P_1}}$} (m-1-2)
(m-3-3) edge node[below] {$\kappa_{L_{P_2}}$} (m-3-2)
(m-1-3) edge node[auto] {$Q_{12}$} (m-3-3);
\end{tikzpicture}
\end{center}
\end{remark}

\subsection{The twisted Mishchenko-Fomenko bundle}
The obstruction $\alpha(M)$ is built out of the Dirac operator on $M$ and a certain bundle of Hilbert $C^*(\pi_1(M))$-modules called the Mishchenko-Fomenko bundle. In the twisted case the algebra as well as the bundle itself have to be replaced by twisted versions. 

\begin{definition}
Let $1 \to \Z/2\Z \to \wh{\pi} \to \pi \to 1$ be a central $\Z/2\Z$-extension of a discrete group $\pi$. Let $C^*_r(\wh{\pi})$ be the reduced group $C^*$-algebra associated to $\wh{\pi}$, likewise let $C^*_{\text{max}}(\wh{\pi})$ be its maximal group $C^*$-algebra. Let $e \in \Z/2\Z$ be the non-trivial element and observe that $q = \tfrac{1}{2}(1 - e) \in C_{r/\text{max}}^*(\wh{\pi})$ is a central projection. Let $C^*_r(\wh{\pi} \to \pi) = q\,C^*_r(\wh{\pi}) = q\,C^*_r(\wh{\pi})\,q$ and define $\Cmax{\wh{\pi}}{\pi}$ analogously. It is straightforward to check that these are in fact $C^*$-algebras, which we will call the \emph{reduced}, respectively \emph{maximal twisted group $C^*$-algebra} \cite[Definition 8.1]{paper:StolzConcordance}.
\end{definition}

\begin{remark}
An extension like in the above definition can be classified by a group $2$-cocycle $c_{\wh{\pi}} \in H^2_{\text{gr}}(\pi,\Z/2\Z)$, which yields an alternative description of $C^*_r(\wh{\pi} \to \pi)$ using completions of the twisted group ring $\C[\pi, c_{\wh{\pi}}]$ \cite[Definition 2.1 for $A = \C$]{paper:BusbySmith}. The algebra $\Cmax{\wh{\pi}}{\pi}$ has the universal property that any projective representation $\rho \colon \pi \to U(H)$ for the cocycle $c_{\wh{\pi}}$ extends to a $*$-homomorphism $\Cmax{\wh{\pi}}{\pi} \to B(H)$.
\end{remark}

Let $M$ be a smooth oriented manifold with $\dim(M) \geq 3$, such that its universal cover $\wtM$ carries a spin structure. Denote by $P_{SO}$ the oriented frame bundle of $M$ and let $\pi = \pi_1(M)$. Let $P_{SO(\wtM)}$ be the frame bundle of the universal cover. Consider the exact sequence
\begin{equation} \label{eqn:ex_seq_wtM}
	\pi_2(\wtM) \to \pi_1(SO(n)) \to \pi_1(P_{SO(\wtM)}) \to 1
\end{equation}
The map $\pi_2(\wtM) \to \pi_1(SO(n))$ sends $f \colon S^2 \to \wtM$ to the transition map $\varphi_f \colon S^1 \to SO(n)$ obtained from the pullback $f^*P_{SO(\wtM)}$. If $\wtM$ carries a spin structure, then $\varphi_f$ factors through $S^1 \to \text{Spin}(n)$ and is therefore nullhomotopic. Comparing (\ref{eqn:ex_seq_wtM}) with the exact sequence
\[
	\pi_2(M) \to \pi_1(SO(n)) \to \pi_1(P_{SO}) \to \pi_1(M) \to 1
\]
for the bundle $P_{SO} \to M$, a diagram chase shows that $\pi_2(M) \to \pi_1(SO(n))$ is also trivial in this case and we obtain a central $\Z/2\Z$-extension \cite{paper:StolzConcordance}:
\begin{equation}\label{eqn:pihat}
	1 \to \Z / 2\Z \to \pi_1(P_{SO}) \to \pi \to 1\ .
\end{equation}

\begin{definition}\label{def:LMF}
The lifting bundle gerbe $L_{\wh{\pi}} \to \wtM^{[2]}$ associated to the above central ex\-ten\-sion~(\ref{eqn:pihat}) is called the \emph{Mishchenko-Fomenko bundle gerbe}.
\end{definition}

Note that a trivialization of $L_{\wh{\pi}} \to \wtM^{[2]}$ \emph{as a line bundle} is given by a split $\pi \to \pi_1(P_{SO}) = \wh{\pi}$. After applying this trivialization, the bundle gerbe multiplication over points $\tilde{m}_1, \tilde{m}_2, \tilde{m}_3$ with $\tilde{m}_2 = \tilde{m}_1g_1$, $\tilde{m}_3 = \tilde{m}_2g_2$ for some $g_1,g_2 \in \pi$ is transformed to $\mu((\tilde{m}_1,\tilde{m}_2, z_1), (\tilde{m}_2, \tilde{m}_3, z_2)) = (\tilde{m}_1, \tilde{m}_3, z_1\,z_2\,c_{\wh{\pi}}(g_1,g_2))$ for some cocycle $c_{\wh{\pi}} \colon \pi \times \pi \to \Z/2\Z$ representing the extension in the group $H^2(\pi,\Z/2\Z)$ and $z_i \in \C$.

\begin{definition}\label{def:MFbundle}
Let $\wh{\pi} = \pi_1(P_{SO})$ and let $1 \to \Z/2\Z \to \wh{\pi} \to \pi \to 1$ be the extension described above. The canonical twisted Hilbert $A$-module bundle for $A = C^*_{\text{max}}(\wh{\pi} \to \pi)$ given by $\Vmax = \wtM \times C^*_{\text{max}}(\wh{\pi} \to \pi)$ is called the \emph{(maximal) twisted Mishchenko-Fomenko bundle}. Similarly, we define $\Vred = \wtM \times C^*_r(\wh{\pi} \to \pi)$.
\end{definition}

Let $E \to \wtM$ be a twisted Hilbert $A$-module bundle for $L_{\wh{\pi}}$. Trivializing the bundle gerbe $L_{\wh{\pi}}$ (and thereby fixing a cocycle $c_{\wh{\pi}}$) turns the twisting into a collection of bundle maps that are on the fibers given by $\gamma^g \colon E_{\tilde{m}} \to E_{\tilde{m}g^{-1}}$ with the property that $\gamma^g \circ \gamma^h = c_{\wh{\pi}}(g,h)\gamma^{gh}$, i.e.\ it behaves like a projective representation. 

Now suppose that $E$ is smooth and comes equipped with a twisted connection $\nabla$. Fix $\widetilde{m} \in \wtM$, let $m$ be its image in $M$ and let $c \colon I \to \wtM$ be a smooth curve in $\wtM$ starting at $\widetilde{m}$. Let $\para{c} \colon I \times E_{\widetilde{m}} \to E$ be the parallel transport with respect to $\nabla$. Let $\wh{g} \in \wh{\pi}$, let $g \in \pi$ be its image under the projection. Observe that $\para{c}$ is equivariant in the sense that $\para{cg^{-1}}(t, \wh{g} \cdot v) = \wh{g} \cdot \para{c}(t,v)$ for $t \in I$, $v \in E$. Let $\tau \colon S^1 \to M$ be a smooth loop in $M$ at $m$ representing an element $h \in \pi = \pi_1(M)$, let $\bar{\tau} \colon I \to \wtM$ be the lift of $\tau$ to $\widetilde{m}$. Note that the endpoint of $\bar{\tau}$ is $\widetilde{m}h^{-1}$. Define $\pEndo{\widetilde{m}}{\tau}(v) := \para{\bar{\tau}}(1,v)$ and let  $\hol{\tau}{\widetilde{m}} \colon E_{\widetilde{m}} \to E_{\widetilde{m}}$ be given by $\hol{\tau}{\widetilde{m}} = \gamma^{h^{-1}} \circ \pEndo{\widetilde{m}}{\tau}$. It satisfies
\begin{align*}
c_{\wh{\pi}}(g,h)^{-1}\,\hol{\bar{m}}{\sigma \ast \tau} &= c_{\wh{\pi}}(g,h)^{-1} \gamma^{(gh)^{-1}} \pEndo{\bar{m}}{\sigma \ast \tau}  = \gamma^{h^{-1}} \circ \gamma^{g^{-1}} \circ \pEndo{\bar{m}g^{-1}}{\tau} \circ \pEndo{\bar{m}}{\sigma} \\
& = \gamma^{h^{-1}} \circ \left(\gamma^{g^{-1}} \circ \pEndo{\bar{m}g^{-1}}{\tau} \circ \gamma^g\right) \circ \gamma^{g^{-1}} \circ \pEndo{\bar{m}}{\sigma} \\
& = \gamma^{h^{-1}} \circ \pEndo{\bar{m}}{\tau} \circ \gamma^{g^{-1}} \circ \pEndo{\bar{m}}{\sigma} = \hol{\bar{m}}{\tau} \circ \hol{\bar{m}}{\sigma}\ .
\end{align*}
which proves that it is a projective representation of $\pi$ on $E_{\widetilde{m}}$ with respect to $c_{\wh{\pi}}$.
\begin{definition} \label{def:projhol}
We will call $\hol{\cdot}{\widetilde{m}}$ the \emph{projective holonomy of $\nabla$  at $\widetilde{m}$}. 
\end{definition}

\section{Index Theory on Twisted Bundles}

\subsection{The Chern Character}
Twisted versions of the Chern character have been developed in \cite[Section 6]{paper:KTheoryBGM}, \cite[Section 3.2]{paper:GerbesAndIndexThms}. Crucial in both cases is the idea that the endomorphism bundle $\Endo{F} \to Y$ of a twisted vector bundle $F$ for a bundle gerbe $L$ descends to a vector bundle $\bgend{F}$ over $M$. Let $\nabla^F$ be a twisted connection on $F$. Since we deal with the case of real bundle gerbes, where curving forms do not matter, the curvature descends to a form $\Omega_F \in \Omega^2(M, \bgend{F})$.
\begin{definition}\label{def:ChernCharacter}
The \emph{twisted Chern character} $\ch(F)$ is the class 
\begin{equation} \label{eqn:ChernCh}
	\ch(F) = {\rm tr}\left(\exp\left(\frac{i\Omega_F}{2\pi}\right)\right) \in H^{\text{even}}(M, \R)\ .
\end{equation}
\end{definition}

It induces a group homomorphism $\ch \colon K^0_{L,\C}(M) \to H^{\text{even}}(M, \R)$. If $L_1, L_2$ are two bundle gerbes with $dd(L_1) = dd(L_2)$ and $Q$ is a trivialization of $L_1^* \boxtimes L_2$, then 
\begin{center}
\begin{tikzpicture}
\matrix (m) [matrix of math nodes, row sep=0.2cm, column sep=1.3cm, 
             text height=1.5ex, text depth = 0.25ex]
{ K^0_{L_1,\C}(M) \\
& H^{\text{even}}(M, \R) \\
K^0_{L_2,\C}(M) \\
};
\path[->,font=\scriptsize]
(m-1-1) edge node[above] {$\ch$} (m-2-2)
(m-3-1) edge node[below] {$\ch$} (m-2-2)
(m-1-1) edge node[auto] {$Q$} (m-3-1);
\end{tikzpicture}
\end{center}
commutes, since $Q$ is a flat line bundle. 

Choose a lifting bundle gerbe $L_P$ with $dd(L) = dd(L_P)$. The K\"unneth theorem and the canonical isomorphism $K^0_{L_P,\C}(M) \to K^0(C(M, \mathcal{K}))$ allows us to extend the above definition to a $K_0(A) \otimes \R$-valued Chern character via 
\[
	\ch \colon K^0_{L,A}(M) \to K^0_{L_P, \C}(M) \otimes K_0(A) \otimes \R \to H^{\text{even}}(M, K_0(A) \otimes \R)\ .
\]
The first map is obtained from Theorem \ref{thm:Twisted_K} using a trivialization $Q$ of $L^* \boxtimes L_P$. As we have seen in Remark \ref{rem:compatibility}, the result of $\ch$ does not depend on the choice of $P$ and $L_P$.

\begin{lemma} \label{lem:chern_mult}
Let $L_i \to Y_i^{[2]}$ for $i \in \{1,2\}$ be bundle gerbes. The following diagram commutes:
\begin{center}
\begin{tikzpicture}
\matrix (m) [matrix of math nodes, row sep=1.3cm, column sep=1.3cm, 
             text height=1.5ex]
{ K^0_{L_1,A}(M) \times K^0_{L_2,\C}(M) & K^0_{L_1 \boxtimes L_2,A}(M) \\
  H^{{\rm even}}(M, K_0(A) \otimes \R) \times H^{{\rm even}}(M,\R) & H^{{\rm even}}(M, K_0(A) \otimes \R) \\
};
\path[->,font=\scriptsize]
(m-1-1) edge node[auto] {$\boxtimes$} (m-1-2)
(m-2-1) edge node[auto] {$\wedge$} (m-2-2)
(m-1-1) edge node[auto] {$\ch \times \ch$} (m-2-1)
(m-1-2) edge node[auto] {$\ch$} (m-2-2);
\end{tikzpicture} 
\end{center}
\end{lemma}

\begin{proof} 
For $A = \C$ the proof is analogous to the untwisted case. For arbitrary~$A$, choose a principal $PO(n)$-bundle $P$ and a trivialization $Q$ of $L_1^* \boxtimes L_P$ as in Theorem~\ref{thm:Twisted_K}. Note that $Q$ and the canonical trivialization of $L_2^* \boxtimes L_2$ induce a trivialization of $(L_1 \boxtimes L_2)^* \boxtimes (L_P \boxtimes L_2)$. The statement is a consequence of the following commutative diagram
\begin{center}
\begin{tikzpicture}
\matrix (m) [matrix of math nodes, row sep=0.7cm, column sep=1cm, 
             text height=1.5ex]
{ K^0_{L_1,A}(M) \times K^0_{L_2,\C}(M) & K^0_{L_P,\C}(M) \otimes K_0(A) \otimes \R \times K^0_{L_2,\C}(M) \\
  K^0_{L_1 \boxtimes L_2,A}(M) & K^0_{L_P \boxtimes L_2,\C}(M) \otimes K_0(A) \otimes \R \\
};
\path[->,font=\scriptsize]
(m-1-1) edge (m-1-2)
(m-2-1) edge (m-2-2)
(m-1-1) edge (m-2-1)
(m-1-2) edge (m-2-2);
\end{tikzpicture} 
\end{center}
in which the vertical maps are given by tensor products and the horizontal ones are the decomposition maps from the K\"unneth theorem.
\end{proof}

\subsubsection{Chern character and traces} 
\label{ssub:chern_character_and_traces}
If $A$ comes equipped with a trace $\tau$, which we will assume for the rest of this section, there is a more direct approach to the Chern character, which was explored in the untwisted case in \cite[Section 4]{paper:SchickKKConnections}. If $V$ is a finitely generated projective Hilbert $A$-module, then we have $\Endo{V} = \mathcal{K}(V,V) \cong V \otimes_A \mathcal{K}(V,A)$, where $\mathcal{K}(V,W)$ denotes the compact adjointable $A$-linear operators. Setting $\tau_V(v \otimes T) = \tau(T(v))$ on elementary tensors, extends the trace to $\Endo{V}$. 

Let $L$ be a bundle gerbe and $E$ be a finitely generated projective twisted Hilbert $A$-module bundle for $L$. As was already noted above, $\Endo{E}$ descends to a bundle of $C^*$-algebras $\bgend{E}$ over $M$. Extending $\tau$ to the fibers yields a linear map $\Omega^n(M, \bgend{E}) \to \Omega^n(M)$ on $\bgend{E}$-valued forms, which we will again denote by~$\tau$.

\begin{definition}\label{def:tauChern}
Let $M$ be a compact smooth manifold, $L$ be a bundle gerbe and $E$ be a twisted Hilbert $A$-module bundle for $L$ equipped with a twisted connection $\nabla^E$ that has curvature $\Omega_E \in \Omega^2(M, \bgend{E})$. As in  \cite[Lemma~4.2]{paper:SchickKKConnections} it follows that 
\[
	\tau\left( \exp\left( \frac{i\Omega_E}{2\pi}\right)\right) \in \Omega^{\text{even}}(M)
\]
is a closed form, which is independent of the choice of twisted connection. Its cohomology class $\ch_{\tau}(E)$ is called the \emph{$\tau$-Chern character} of $E$.
\end{definition}

To explain the connection between $\ch$ from Definition \ref{def:ChernCharacter} and $\ch_{\tau}$ we need the following concept.
\begin{definition}\label{def:dimension}
Let $V$ be a finitely generated projective right Hilbert $A$-module. Its \emph{dimension} is defined to be $\dim_{\tau}(V) = \tau_V(\id_V)$, where $\tau_V$. If $V \cong t A^n$ for some projection $t \in M_n(A) = M_n(\C) \otimes A$, then $\dim_{\tau}(V) = (\text{tr} \otimes \tau)(t)$. $\dim_{\tau}$ yields a well-defined map $K_0(A) \to \R$.	
\end{definition}

\begin{theorem}\label{thm:chTau_vs_ch}
Let $L$ be a bundle gerbe and let $E$ be a twisted Hilbert $A$-module bundle for $L$. Then we have
\begin{equation} \label{eqn:dimtau}
	\dim_{\tau}(\ch(E)) = \ch_{\tau}(E) \ \in \Heven{M}{\R}\ .
\end{equation}
\end{theorem}

\begin{proof}\label{pf:chTau_vs_ch}
First observe that $\ch_{\tau}$ is still multiplicative, in the sense that for a twisted Hilbert $A$-module bundle $E$ and a twisted vector bundle $F$ we have $\ch_{\tau}(E \boxtimes F) = \ch_{\tau}(E) \cup \ch(F)$. In fact, this property rests only on the fact that the canonical tensor product connection $\nabla^{E \boxtimes F}$ has curvature $\Omega_{E \boxtimes F} = \Omega_E \otimes \id + \id \otimes \Omega_F \in \Omega^2(M, \bgend{E} \otimes \bgend{F})$ together with the calculation
\[
	\tau\left(\exp\left( \frac{i\,\Omega_{E \boxtimes F}}{2\pi} \right)\right) = \tau\left(\exp\left( \frac{i\,\Omega_{E}}{2\pi}\right)\right) \wedge \tau\left(\exp\left( \frac{i\,\Omega_{F}}{2\pi} \right)\right)
\]
which is proven as in the untwisted case. 

Choose $P$ and a trivialization $Q$ as in Theorem \ref{thm:Twisted_K} and let $F$ be a twisted vector bundle for $L_P$. Let $V$ be a finitely generated projective Hilbert $A$-module. Observe that for the trivial bundle $\trivial{V}$ we have $\ch_{\tau}(\trivial{V}) = \dim_{\tau}(V)$. Thus,
\[
	\ch_{\tau}(F \otimes V) = \ch_{\tau}(F \boxtimes \trivial{V}) = \ch(F) \cup \ch_{\tau}(\trivial{V}) = \dim_{\tau}(V)\,\ch(F)\ .
\]
Likewise we have $\ch(F \otimes V) = \ch(F) \otimes [V] \in \Heven{M}{K_0(A) \otimes \R}$. Therefore the following diagram commutes
\begin{center}
\begin{tikzpicture}
\matrix (m) [matrix of math nodes, row sep=0.7cm, column sep=1cm, 
             text height=1.5ex]
{ K^0_{L_P,\C}(M) \otimes K_0(A) & K_0(C(M, \mathcal{K} \otimes A)) & K^0_{L,A}(M) \\
  & \Heven{M}{\R} \\
};
\path[->,font=\scriptsize]
(m-1-1) edge (m-1-2)
(m-1-2) edge node[auto] {$\rho_Q^{-1}$} (m-1-3)
(m-1-1) edge node[below] {$\ch \otimes \dim_{\tau}$} (m-2-2)
(m-1-3) edge node[below] {$\ch_{\tau}$} (m-2-2);
\end{tikzpicture} 
\end{center}
where the first horizontal map is given by the canonical identification $K^0_{L_P,\C}(M) \cong K_0(C(M,\mathcal{K}))$ and the tensor product of modules. The statement follows if $\ch_{\tau}$ vanishes on the second summand in the K\"unneth decomposition. We can identify $K^1_{L_P,\C}(M)$ with $K^0_{\pi^*L_P,\C}(\R \times M)$. Likewise we can represent classes in $K_1(A)$ as compactly supported virtual Hilbert $A$-module bundles over $\R$. If $F = (F_+, F_-, \varphi)$ represents an element in $K^0_{\pi^*L_P,\C}(\R \times M)$ and $W = (W_+,W_-, \psi)$ an element in $K_0(C_0(\R,A))$, then the graded tensor product $F \,\wh{\boxtimes}\,W$ yields an element in $K^0_{\pi^*L_P,A}(\R^2 \times M)$ and Bott periodicity maps it to $b([F \,\wh{\boxtimes}\, W]) \in K^0_{L_P,A}(M)$. It is a consequence of the multiplicativity that $\ch_{\tau}(X) = \ch_{\tau}(b(X)) \cup e$, where $e \in H^2_c(\R^2, \R) \cong \R$ denotes a generator and $X$ is a triple representing a class in $K^0_{\pi^*L_P,A}(\R^2 \times M)$. Thus,
\[
	\ch_{\tau}(b(F \boxtimes W)) \cup e = \ch_{\tau}(F \boxtimes W) = \ch(F) \cup \ch_{\tau}(W) = 0
\]
since $\ch_{\tau}(W) \in H^{\text{even}}_c(\R,\R) = 0$.
\end{proof}

\subsection{Local $C^*$-algebras and twisted bundles} \label{sub:Local_and_Twisted}
A \emph{local $C^*$-algebra} is a $\ast$-algebra $B$ equipped with a pre-$C^*$-norm such that $M_n(B)$ is closed under holomorphic functional calculus for all $n \in \N$. If $\mathcal{B} \to M$ is a locally trivial bundle with fiber $B$ and structure group ${\rm Aut}(B)$ equipped with the point norm topology, then the section algebra $C(M, \mathcal{B})$ is again a local $C^*$-algebra as can be checked directly from the definitions using the naturality of holomorphic functional calculus.

A right $B$-module is called \emph{finitely generated} and \emph{projective} if it is isomorphic to one of the form $t\,B^n$ for a projection $t \in M_n(B)$. The groups $K_i(B)$, $i \in \{0,1\}$ are now defined as in \cite{book:Blackadar}. If $A$ is the completion of $B$, then $K_i(B) \cong K_i(A)$. Twisted $B$-module bundles are defined just as in \ref{def:TwistedHilbertA} and yield analogous groups $K^0_{L,B}(M)$ and $K^1_{L,B}(M) := K^0_{\pi_M^*L,B}(\R \times M)$. Given a twisted $B$-module bundle $\mcal{E}$, the fiberwise inner product $\mcal{E} \otimes_B A$ is a twisted Hilbert $A$-module bundle. Thus we get a homomorphism
$K^0_{L,B}(M) \to K^0_{L,A}(M)$

Let $S$ be as in Theorem \ref{thm:Twisted_K}. The exterior tensor product $S^* \boxtimes \mcal{E}$ involves no completion. Therefore the proof of that theorem still works and gives
\(
	K^0_{L, B}(M) \cong K_0(C(M, \mcal{K} \otimes B))
\), which also proves $K^0_{L,B}(M) \cong K^0_{L,A}(M)$.
The Bott map $b \colon K^0_{L,B}(M) \to K^0_{\pi_M^*L,B}(\R^2 \times M)$ is also still well-defined and a comparison with the corresponding map for its completion shows that $b$ still is an isomorphism. Likewise, there is a K\"unneth homomorphism
\[
	K^0_{L_P,\C}(M) \otimes K_0(B) \otimes \R \oplus K^1_{L_P,\C}(M) \otimes K_1(B) \otimes \R \to K_0(C(M, \mathcal{K} \otimes B)) \otimes \R
\]
that just involves algebraic tensor products in the fibers. Comparison with the completion reveals this to be an isomorphism as well. This enables us to define $\ch \colon K^0_{L,B}(M) \to \Heven{M}{K_0(B) \otimes \R}$.

Let $B$ be a unital local $C^*$-algebra carrying a trace $\tau$. If $\mcal{V}$ is a finitely generated and projective $B$-module, then $\Endo{\mcal{V}}$ coincides with the finite rank operators and we have $\Endo{\mcal{V}} \cong \mcal{V} \otimes_B \Homom{\mcal{V}}{B}$, where the tensor product is a quotient of the algebraic tensor product. Repeating the constructions in Section \ref{ssub:chern_character_and_traces} we can define $\ch_{\tau} \colon K^0_{L,B}(M) \to \Heven{M}{\R}$ and $\dim_{\tau} \colon K_0(B) \to \R$. The proof of Theorem \ref{thm:chTau_vs_ch} shows that $\dim_{\tau}(\ch_Q(\mcal{E})) = \ch_{\tau}(\mcal{E}) \cup \ch(Q)$.

\subsection{The projective Dirac operator}
Let $(M, g)$ be an oriented Riemannian manifold of even dimension $n > 3$. The oriented frame bundle $P_{SO}$ together with the defining central $\Z/2\Z$-extension ${\rm Spin}(n) \to SO(n)$ gives a lifting bundle gerbe $\Lspin \to P_{SO}^{[2]}$. Since we assume $M$ to be even-dimensional, the complex spinor module $\Delta_{\C}$ is $\Z/2\Z$-graded, i.e.\ $\Delta_{\C} = \Delta_+ \oplus \Delta_-$. There is a $\Z/2\Z$-graded twisted vector bundle $S = P_{SO} \times \Delta_{\C}$ with $S_{\pm} = P_{SO} \times \Delta_{\pm}$ for $\Lspin$, where the action is given by $\gamma([\widehat{g}, \lambda] \otimes v) = \lambda\,\wh{g} \cdot v$ for $\wh{g} \in {\rm Spin}(n)$, $\lambda \in \C$ and $v \in \Delta_{\C}$. If we fix a connection form $\eta \in \Omega^1(P_{SO},  \mathfrak{so}(n))$, then there is a canonical twisted connection $\nabla^S$ on $S$, which acts like $\nabla^S(\sigma) = d\sigma + \rho_*(\eta)\cdot \sigma$ for a section $\sigma \in C(P_{SO}, S)$, where $\rho$ is induced by $O(n) \to U(n)$. 

Let $E \to Y$ be a smooth twisted Hilbert $A$-module bundle for a bundle gerbe $L$ with $dd(L) = -dd(\Lspin) = dd(\Lspin)$ equipped with a twisted connection $\nabla^E$. Choose a trivialization $Q$ of $\Lspin \boxtimes L$, then $Q(S \boxtimes E)$ is a Hilbert $A$-module bundle over $M$. Fixing a (flat) connection on $Q$, $\nabla^{S \boxtimes E} = \nabla^S \otimes \id + \id \otimes \nabla^E$ descends to a connection on $Q(S \boxtimes E)$. The latter bundle carries an action of the complex Clifford bundle $\C\ell(M)$. We have $\nabla_X^{S \boxtimes E}(c \cdot \sigma) = \nabla_X^{\C\ell(M)}(c)\cdot \sigma + c\cdot \nabla_X^{S \boxtimes E}(\sigma)$ for $X \in C^{\infty}(M,TM)$, $c \in C^{\infty}(M,\C\ell(M))$, $\sigma \in C^{\infty}(M, Q(S \boxtimes E))$. 
\begin{definition}\label{def:Twisted_Dirac}
Let $S,E$ and $Q$ be as above, then the elliptic first order differential operator 
$D^E \colon C^{\infty}(M,Q(S \boxtimes E)) \to C^{\infty}(M, TM \otimes Q(S \boxtimes E)) \to C^{\infty}(M, Q(S\boxtimes E))$
where the first map is induced by the connection $\nabla^{S \boxtimes E}$ and the metric and the second is Clifford multiplication will be called the \emph{projective Dirac operator} twisted by $E$. 
\end{definition}

The bundle gerbe $\Lspin$ provides a canonical trivialization $Q$ of $\Lspin \boxtimes \Lspin$ with the property that $Q(S^* \boxtimes S) \cong \C \ell(M)$ as a bundle of $\C\ell(M)$-$\C\ell(M)$-bimodules. This identifies the operator $D^S \colon C^{\infty}(M, \C\ell(M)) \to C^{\infty}(M, \C\ell(M))$ with the right Clifford linear Dirac operator. As described in \cite[Section 24.5]{book:Blackadar}, $D^S$ yields a class in the $K$-homology group $KK(C(M,\C\ell(M)), \C)$.

Since the connection is even and Clifford multiplication is an odd operation, $D^E$ has the decomposition $D^E_{\pm} \colon C^{\infty}(M,Q(S_+ \boxtimes E)) \to C^{\infty}(M, Q(S_- \boxtimes E))$.
\begin{theorem}\label{thm:indextheorem}
Let $S,E$ and $Q$ be as in the last paragraph and denote by $D^E_+$ the positive part of the projective Dirac operator twisted by $E$, then 
\[
{\rm ind}(D_+^{E}) = \int_M \wh{A}(M) \ch(E)\ .
\]
In particular, the index of $D_+^E$ does not depend on the choice of $Q$.
\end{theorem}

\begin{proof}\label{pf:indextheorem}
Let $\pi_M \colon T^*M \to M$ be the bundle projection and denote by $D(T^*M)$ and $S(T^*M)$ the disc and sphere bundle of the cotangent bundle. The symbol $\sigma \colon \pi_M^*Q(S_+ \boxtimes E) \to \pi_M^*Q(S_- \boxtimes E)$ of $D_+^{E}$ yields an element $[Q(S_+ \boxtimes E), Q(S_- \boxtimes E), \sigma] \in K^0_A(D(T^*M), S(T^*M)) \cong K^0_A(T^*M)$. Denote by 
\[
	\phi \colon H^{\ast}(M, \R) \to H^{\ast + \dim(M)}(D(T^*M), S(T^*M); \R)
\]
the Thom isomorphism. From the Mishchenko-Fomenko index theorem \cite{paper:MS_index_thm} together with Lemma~\ref{lem:chern_mult} we obtain
\begin{align*}
	{\rm ind}(D_+^{E}) & = \int_M {\rm Td}(M)\,\phi^{-1}\left(\ch([Q(S_+ \boxtimes E), Q(S_- \boxtimes E), \sigma])\right) \\ 
& = \int_M {\rm Td}(M)\,\phi^{-1}\left(\ch([S_+, S_-,\sigma_D])\right)\,\ch(E) = \int_M \widehat{A}(M)\ch(E)
\end{align*}
where the identification of ${\rm Td}(M)\,\phi^{-1}\left(\ch([S_+, S_-,\sigma_D])\right)$ with the $\widehat{A}$-genus described in \cite{book:LawsonMichelsohn} works in the twisted case as well. 
\end{proof}

\begin{theorem} \label{thm:KKproduct}
Let $D$ be the projective Dirac operator. Let $E \rightarrow Y$ be a finitely generated projective twisted Hilbert $A$-module bundle for a bundle gerbe $L$ with $dd(L) = dd(\Lspin)$. Let $Q$ be a trivialization of $L^* \boxtimes \Lspin$ and $[E^S] := \rho_Q([E]) = [C(M,Q(S^* \boxtimes E))] \in KK(\C, C(M,\C\ell(M) \otimes A))$. Then
\(
	[E^S] \otimes_{C(M, \C\ell(M))} [D^S] \in KK(\C, A) \cong K_0(A)
\)
is the class representing the Mishchenko-Fomenko index of $D^{E}$.
\end{theorem}

\begin{proof}
It suffices to show that the intersection product coincides with the Fredholm module:
\[
	\left[L^2(Q(S^* \boxtimes E)), D^{E}\left(1 + \left(D^{E}\right)^2\right)^{-\frac{1}{2}} \right] \in KK(\C, A)\ . 
\]
The proof of this is just a slight modification of \cite[Theorem~5.22]{paper:SchickKKConnections}, therefore we omit it.
\end{proof}

From this decomposition of the index in twisted $K$-theory, we obtain the following naturality result in analogy to the untwisted case \cite[Lemma 3.1]{paper:HankeSchick}.

\begin{corollary}\label{cor:naturality}
Let $D$ be the projective Dirac operator and let $E \to Y$ be a twisted Hilbert $A$-module bundle for a bundle gerbe $L$ with $dd(L) = dd(\Lspin)$. Given a $C^*$-algebra homomorphism $\varphi : A \to B$ define the twisted Hilbert $B$-module bundle~$F$ via
\(
	F = E \otimes_\varphi B
\).
Then: $\varphi_*\!\left(\left[D^{E}\right]\right) = \left[D^{F}\right]$, where $\varphi_* : K_0(A) \to K_0(B)$ denotes the induced map on $K$-theory.
\end{corollary}

\begin{proof}
Let $[E^S] := [C(M,Q(S^* \boxtimes E))] \in KK(\C, C(M,\C\ell(M) \otimes A))$ and choose a trivialization $Q$ of $L^* \boxtimes \Lspin$. The naturality of the Kasparov product yields
\begin{eqnarray*}
\varphi_*\!\left(\left[D^{E}\right]\right) &=& \varphi_*\!\left([E^S] \otimes_{C(M, \C\ell(M))} [D^{S}]\right) \\
&=& \left((\text{id}_{C(M,\C\ell(M))} \otimes \varphi_*) [C(M,Q(S^* \boxtimes E))] \right) \otimes_{C(M, \C\ell(M))} [D^{S}] \\ 
&=& [C(M,Q(S^* \boxtimes F))] \otimes_{C(M, \C\ell(M))} [D^{S}] = \left[D^{F}\right]\ . \qedhere
\end{eqnarray*}
\end{proof}

\section{Enlargeable manifolds with spin universal cover} 
In this section we will extend the result of \cite{paper:HankeSchick, paper:HankeSchickInfinite} about the Rosenberg index obstruction for enlargeable spin manifolds to enlargeable orientable manifolds with spin structure on the universal cover. Since $\wtM$ may be non-compact, we have to work in a twistedly equivariant setting. 

\begin{definition}\label{def:enlargeable}
A connected closed oriented manifold $M$ with fixed metric $g$ is called \emph{enlargeable} if the following holds: For every $\varepsilon > 0$, there is a connected cover $\bar{M} \to M$ carrying a spin structure and an $\varepsilon$-contracting map 
\[
	\bar{\psi}_{\varepsilon} \colon (\bar{M},\bar{g}) \to (S^n, g_0) 
\]
which is constant outside a compact subset of $\bar{M}$ and of nonzero degree. Here, $\bar{g}$ denotes the induced metric on $\bar{M}$ and $g_0$ is the standard metric on $S^n$.
\end{definition} 

\begin{definition} \label{def:index_obstruction}
Let $M$ be a closed even-dimensional oriented Riemannnian manifold and let $\mathcal{V}_{\rm max}$ be the maximal twisted Mishchenko-Fomenko bundle (Def.\ \ref{def:MFbundle}). The \emph{(maximal) twisted Stolz-Rosenberg obstruction} is given by 
\[
	\theta^{\rm max}(M) = {\rm ind}(D^{\mathcal{V}_{\rm max}}) \in K_0(C^*_{\rm max}(\wh{\pi}\to \pi)).
\]
\end{definition}

\subsection{Almost flat twisted bundles}
In our proof of the non-vanishing of $\theta^{\rm max}(M)$ for enlargeable $M$ with spin universal cover we will extend the argument given in \cite{paper:HankeSchick, paper:HankeSchickInfinite} to the twisted setting. Thus, it will rely on almost flat \emph{twisted} bundles.

\begin{definition}
Let $L_{\wh{\pi}}$ be the Mishchenko-Fomenko bundle gerbe. A sequence $E_i \to \wtM$, $i \in \N$ of smooth twisted vector bundles for $L_{\wh{\pi}}$ equipped with connections $\nabla^i$ will be called \emph{a sequence of almost flat twisted bundles}, if 
\(
	\lim_{i \to \infty} \lVert \Omega_i \rVert = 0
\),
where $\Omega_i$ is the curvature of the connection $\nabla^i$. The norm is induced by the pointwise norm on $\bgend{E_i} \to M$ and the maximum norm on the unit sphere bundle in $\Lambda^2(M)$. Moreover, the twistings $\gamma_i^g \colon E_i \to g^*E_i$ considered as sections $C(\wtM, \Homom{E_i}{g^*E_i})$ should be locally Lipschitz continuous with a global Lipschitz constant $C$ independent of $i$, i.e.\ each point $\widetilde{m} \in \wtM$ has an open neighborhood $U$ of $\widetilde{m}$, such that $E_i$ and $g^*E_i$ are trivial over $U$ and such that $\left.\gamma_i^g\right|_{U}$ viewed as an element in $C(U, U(V))$ ($V$ being the typical fiber of $E_i$) satisfies
\[
	\lVert \gamma_i^g(\widetilde{m}_1) - \gamma_i^g(\widetilde{m}_2) \lVert \leq C\,d(\widetilde{m}_1, \widetilde{m}_2)
\]
where the metric is induced by the Riemannian structure pulled back to $\wtM$. 
\end{definition}

Let $M$ be an orientable smooth manifold and let $\bar{M}$ be a cover of $M$ equipped with a spin structure, let $\pi = \pi_1(M)$ and $\bar{\pi} = \pi_1(\bar{M}) \subset \pi$. The bundle projection $P_{\text{Spin}(\bar{M})} \to \bar{M}$ of the principal $\text{Spin}(n)$-bundle induces an isomorphism $\pi_1(P_{\text{Spin}(\bar{M})}) \to \bar{\pi}$, which in turn yields an injective group homomorphism $\sigma \colon \bar{\pi} \to \pi_1(P_{\text{Spin}(\bar{M})}) \to \pi_1(P_{\text{SO}(\bar{M})}) \to \wh{\pi}$, where the first map is the inverse of the above isomorphism and the other two homomorphism arise from the respective coverings. We call a (set-theoretic) split $s \colon \pi \to \wh{\pi}$ \emph{compatible}, if $s(gh) = s(g)\sigma(h)$ for $g \in \pi$ and $h \in \bar{\pi}$.

\begin{lemma}\label{lem:projSpin}
Let $M$, $\bar{M}$, $\pi$ and $\bar{\pi}$ be as in the last paragraph. The choice of a compatible set-theoretic split $s \colon \pi \to \wh{\pi}$ turns the Hilbert space $H = \ell^2(\pi / \bar{\pi})$ into a projective representation of $\pi$, which corresponds to an honest representation of $\wh{\pi}$.
\end{lemma}

\begin{proof}\label{pf:projSpin}
Let $e \in \wh{\pi}$ be the image of the non-trivial element of $\Z/2\Z$ and let $q = \tfrac{1}{2}(1 -e)$. Since $q$ is central in the group ring, it yields a well-defined projection $q \colon \ell^2(\wh{\pi}/\sigma(\bar{\pi})) \to \ell^2(\wh{\pi}/\sigma(\bar{\pi}))$. The space $q\,\ell^2(\wh{\pi}/\sigma(\bar{\pi}))$ is a representation of $\wh{\pi}$, on which $e$ acts by multiplication with $-1$, i.e.\ a projective representation of $\pi$. 

A compatible split now yields an injective map $\pi/\bar{\pi} \to \wh{\pi}/\sigma(\bar{\pi})$ via $[g] \mapsto [s(g)]$. It induces an isometric isomorphism $\ell^2(\pi/\bar{\pi}) \to q\,\ell^2(\wh{\pi}/\sigma(\bar{\pi}))$
\end{proof}

\begin{theorem} \label{thm:almostflat}
Let $M$ be an even-dimensional orientable manifold that is enlargeable in the sense of Definition \ref{def:enlargeable}. Let $i \in \N$ be a positive natural number. Then there is a $C^*$-algebra $C_i$ (which will be constructed in the proof) and a twisted Hilbert $C_i$-module bundle $E_i \to \wtM$ for the Mishchenko-Fomenko bundle gerbe $L_{\wh{\pi}}$ together with a twisted connection $\nabla_i$ that has the following properties: The curvature $\Omega_i$ of $E_i$ satisfies 
\[
	\lVert \Omega_i \rVert \leq \frac{1}{i}\,C
\]
where $C$ is a constant depending only on $\dim(M)$. Moreover, there is a split extension 
\(
	0 \to \K \to C_i \to X_i \to 0
\)
with a certain $C^*$-algebra $X_i$. In particular, each $K_0(C_i)$ splits off a $\Z = K_0(\K)$ summand and the $K_0(\K)$-part of the index of the projective Dirac operator $D_+^{E_i}$ is different from $0$.
\end{theorem}

\begin{proof}
Let $2n = \dim(M)$ and $\pi = \pi_1(M)$. Since the Chern character is ratio\-nally an isomorphism, there is a vector bundle $F \to S^{2n}$ with non-vanishing top Chern class $c_n(F) \neq 0$. Choose a connection $\eta_F$ on $F$ and fix $i \in \N$. Since $M$ is enlargeable, there exists a spin covering $\bar{M} \to M$ together with a $\frac{1}{i}$-contracting map
\(
	\psi \colon \bar{M} \to S^{2n}\ ,
\)
which is constant outside a compact subset $K$ of $\bar{M}$. Let $P_F$ be the principal $U(n)$-bundle of frames in $F$. Since $\psi$ is constant on $M \backslash K$ we can choose a trivialization for the principal $U(n)$-bundle $\psi^*P_F$ over this set: 
\[
	\left.(\psi^*P_F)\right|_{M \backslash K} \cong (M \backslash K) \times U(n)
\]
such that the pullback connection $\left.\psi^*\eta_F\right|_{M \setminus K}$ is flat. Let $\rho \colon \bar{M} \to M$ be the covering map, $\widetilde{\rho} \colon \wtM \to M$ the universal cover and $\bar{\pi} = \pi_1(\bar{M})$. As described in the proof of \cite[Proposition 1.5]{paper:HankeSchickInfinite} we can cover $M$ by open sets $U_j, j \in I$, such that each component $V_{\lambda, j} \subset \bar{M}$ of $\rho^{-1}(U_j)$ maps diffeomorphically onto $U_j$, intersects only one component $V_{\mu, k}$ of $\rho^{-1}(U_k)$ for any $k$ and such that $\left.\psi^*P_F\right|_{V_{\lambda,j}}$ trivializes. Let $J_j = \pi_0(\rho^{-1}(U_j))$ be the index set labeling the components, likewise set $\widetilde{J}_j = \pi_0(\widetilde{\rho}^{-1}(U_j))$. Let $\widetilde{\varphi}_{\alpha,j} \colon \widetilde{J}_j \to \pi / \bar{\pi}$ be the map that sends $\alpha g$ to $[g] \in \pi / \bar{\pi}$ for $g \in \pi$, where $\pi$ acts on $\widetilde{J}_j$ by deck transformations. Since $\bar{M} = \wtM / \bar{\pi}$, this induces bijections $\varphi_{\lambda,j} \colon J_j \to \pi / \bar{\pi}$ for each $\lambda \in J_j$. Note that 
\begin{equation}\label{eqn:equivariance}
	\varphi_{[\alpha g],j} = g^{-1} \cdot \varphi_{[\alpha],j}\ .
\end{equation}	
Moreover, if $\lambda \in J_j$ and $\mu \in J_k$ belong to components with non-empty intersection, then $\varphi_{\lambda,j}(\kappa) = \varphi_{\mu, k}(\tau)$ if $\tau$ and $\kappa$ intersect. 

Consider the Hilbert space
\(
	H = \ell^2(\pi / \bar{\pi}) \otimes \C^n
\). 
Let $C_S \subset \mathcal{B}(H)$ be the $C^*$-algebra generated by the group of all permutations of $\pi / \bar{\pi}$ \emph{and} all multiplications by functions $f \colon \pi / \bar{\pi} \to S^1$. So we have permutation operators with $S^1$-entries instead of just $1$s as a generating set of $C_S$. Let $C_T \subset \mathcal{B}(H)$ be $C^*$-algebra generated by linear transformations, which are of the form
\[
	T \colon H \to H \quad ; \quad T([g] \otimes v) = [h] \otimes T'v \quad \text {and} \quad \left.T\right|_{([g] \otimes \C^n)^{\perp}} = 0\ .
\]
for some matrix $T' \in M_n(\C)$ and $[g], [h] \in \pi/{\bar{\pi}}$. Let $C_{S,T}$ be the $C^*$-algebra generated by $C_S$ and $C_T$ inside of $\mathcal{B}(H)$ and note that $C_T$  is a $2$-sided ideal in $C_{S,T}$. Moreover, $C_T$ is isomorphic either to the compact operators or to a matrix algebra. Applying the stabilization trick of \cite{paper:HankeSchickInfinite} we can without loss of generality assume that the former is the case. Let $C_i = \left\{ (c_1, c_2) \in C_{S,T} \times C_{S,T} \ |\ c_1 - c_2 \in C_T\right\}$. This algebra fits into a split exact sequence
\(
	0 \to C_T \to C_i \to C_{S,T} \to 0
\)
with the splitting induced by the diagonal map, $C_T \to C_i$  via $a \mapsto (a,0)$ and $C_i \to C_{S,T}$ via $(a,b) \mapsto b$. 

We choose trivializations of $\psi^*P_F$ over the sets $V_{\lambda,j}\subset \bar{M}$, where we take the trivialization fixed above if $V_{\lambda,j}$ is a subset of $\bar{M} \backslash K$. This way we get a cocycle on the double intersections $V_{(\lambda, \mu), (j,k)} = V_{\lambda, j} \cap V_{\mu, k}$:
\[
	T'_{(\lambda, \mu), (j,k)} \colon V_{(\lambda,\mu), (j,k)} \to U(n)\ .
\]
We can extend $T'_{(\lambda, \mu), (j,k)}$ to a cocycle with values in the unitary group $U(C_{S,T})$ as follows:
\begin{align*}
	T^{1}_{(\lambda, \mu), (j,k)}(x)(\varphi_{\lambda,j}(\kappa) \otimes v) & = \varphi_{\mu,k}(\tau) \otimes T'_{(\kappa, \tau), (j,k)}(\bar{x})(v) \ , 
\end{align*}
where $\tau \in \pi / \bar{\pi}$ is the index of the component of $\rho^{-1}(U_k)$ that intersects $V_{\kappa, j}$ and $\bar{x}$ denotes the lift of $\rho(x)$ to the component $V_{\kappa, j}$. This map actually does nothing to the first tensor factor by our previous considerations. Let $T^2_{(\lambda, \mu), (j,k)}$ be the constant map with value $1 \in U(C_{S,T})$. $T'_{(\kappa, \tau), (j,k)}$ is different from the identity only for finitely many pairs $(\kappa, \tau)$. Thus, 
\[
	T_{(\lambda, \mu), (j,k)} \colon V_{(\lambda,\mu), (j,k)} \to U(C_i) \quad ; \quad T_{(\lambda, \mu), (j,k)} = (T^{1}_{(\lambda, \mu),(j,k)},  T^{2}_{(\lambda, \mu),(j,k)})\ .
\]
is a well-defined cocycle with values in $U(C_i)$. We therefore get a smooth Hilbert $C_i$-module bundle $\bar{E}_i \to \bar{M}$, whose pullback to $\wtM$ will be $E_i = \wtM \times_M \bar{E}_i$. By Lemma~\ref{lem:projSpin}, the space $\ell^2(\pi / \bar{\pi})$ carries a projective unitary representation of $\pi$, which induces a projective representation $r \colon \pi \to U(C_i)$. For $\alpha, \beta \in \widetilde{J}_j$  denote the corresponding components of $\widetilde{\rho}^{-1}(U_j)$ by $W_{\alpha,j}$ and $W_{\beta,j}$ respectively. Let $\lambda = [\alpha]$, $\mu = [\beta]$, $\lambda' = [\alpha g^{-1}]$ and $\mu' = [\beta g^{-1}] \in J_j$. We define
\[
	\gamma^g \colon W_{\alpha,j} \times C_i \to W_{\alpha g^{-1}, j} \times C_i
\]
by left multiplication with $r(g)$. Due to equation (\ref{eqn:equivariance}) and with $\varphi_{\lambda, j}(\kappa) = \varphi_{\mu, k}(\tau) = [h] \in \pi / \bar{\pi}$ we have
\begin{align*}
	\left(T_{(\lambda', \mu'),(j,k)}(x) \cdot r(g)\right)\left( \varphi_{\lambda,j}(\kappa) \otimes v\right)  & = c_{\wh{\pi}}(g,h)\,\varphi_{\mu',j}(\tau) \otimes T'_{(\kappa, \tau),(j,k)}(v) \\
	& = \left(r(g) \cdot T_{(\lambda, \mu),(j,k)}(x)\right)\left( \varphi_{\lambda,j}(\kappa) \otimes v\right)\ .
\end{align*}
Thus, $\gamma^g$ intertwines the transition functions of $E_i$ and $g^*(E_i)$. Therefore it yields a well-defined twisting map
\(
	\gamma^g \colon E_i \to g^*(E_i)
\).
This clearly satisfies the Lipschitz condition, since it even is locally constant. 

Let $\eta_{\kappa, j} \in \Omega^1(V_{\kappa,j}, \mathfrak{u}(n))$ be the pullback of $\eta_F$ via the trivialization. These induce forms in $\Omega^1(V_{\lambda,j}, C_{S,T}^a)$, where $C_{S,T}^a$ denotes the anti-selfadjoint operators in $C_{S,T}$, via  
\[
	\left(\eta^{E_i}_{\lambda,j}\right)_x(\xi)(\varphi_{\lambda,j}(\kappa) \otimes v) = \varphi_{\lambda,j}(\kappa) \otimes \left(\eta_{\kappa, j}\right)_x(\xi)\cdot v\ .
\]
Since $\eta_{\kappa,j}$ is non-zero only for finitely many $\kappa$, we can extend $\eta^{E_i}_{\lambda,j}$ to a well-defined $1$-form with values in the anti-selfadjoint operators of $C_i$ by setting it to zero in the second component. These $1$-forms inherit their transformation behaviour from the forms $\eta_{\kappa,j}$. Thus, they yield a $C_i$-linear connection $\nabla^i$ on sections of $E_i$. Just like above it follows from (\ref{eqn:equivariance}) that $\nabla^i$ is a twisted connection. Since the norm of the curvature $\Omega_i$ of $\nabla^i$ coincides with that of $\psi^* \Omega_F$, we have
\[
	\lVert \Omega_i \rVert = \lVert \psi^*\Omega_F \rVert \leq \frac{1}{i}\,C\ .
\]
It remains to be shown that the $K_0(\K)$-part of ${\rm ind}(D_+^{E_i})$ does not vanish. Here we proceed exactly as in \cite{paper:HankeSchickInfinite}: Let $\mcal{T} \subset C_T \cong \K$ be the trace class ideal and let $D_i$ be the algebra given by
\[
	D_i = \left\{ (c_1,c_2) \in C_{S,T} \times C_{S,T} \ | \ c_1 - c_2 \in \mcal{T} \right\}\ .
\]
Since the proof of Lemma \cite[lemma 2.4]{paper:HankeSchickInfinite} applies to $D_i$ with the changed $C_{S,T}$ as well, $D_i$ is a unital local $C^*$-algebra with a trace $\tau(c_1,c_2) = {\rm tr}(c_1 - c_2)$, which coincides with the trace of the element after projecting it from $C_i$ to $\mcal{T}$. Its $C^*$-completion is $C_i$. Since $K_0(C_i) \cong K_0(D_i)$, we can extend $\dim_{\tau}$ from Definition~\ref{def:dimension} to a functional on $K_0(C_i)$ and it suffices to prove that $\dim_{\tau}({\rm ind}(D_+^{E_i})) \neq 0$. The transition functions in the definition of $E_i$ actually take values in $U(D_i)$ and thus lead to a twisted $D_i$-module bundle $\mcal{E}_i$ in the sense of Section \ref{sub:Local_and_Twisted}. By Theorem~\ref{thm:indextheorem} we have
\begin{align*}
	\dim_{\tau}({\rm ind}(D_+^{E_i})) & = \int_M \wh{A}(M) \dim_{\tau}(\ch(E_i))  \\
	& = \int_M \wh{A}(M) \dim_{\tau}(\ch(\mcal{E}_i)) = \int_M \wh{A}(M) \ch_{\tau}(\mcal{E}_i)\ .
\end{align*}
We can identify $\Omega_{\mcal{E}_i} \in \Omega^2(M, \bgend{\mcal{E}_i})$ with an equivariant form in $\Omega^2(\wtM, \Endo{\mcal{E}_i})$. If we carry out the integration over a single subset $U_j \subset M$, we could integrate instead over the subset $W_{\alpha,j} \subset \wtM$ for some $\alpha \in \widetilde{J}_j$. This is independent of the choice of $\alpha$ by equivariance. But over $W_{\alpha,j}$ the form $\left.\tau(\Omega_{\mcal{E}_i} \wedge \dots \wedge \Omega_{\mcal{E}_i})\right|_{W_{\alpha,j}} \in \Omega^{\rm even}(W_{\alpha,j}, \R)$ coincides with the sum of all $\left.\Omega_{\psi^*F} \wedge \dots \wedge \Omega_{\psi^*F}\right|_{V_{\kappa,j}} \in \Omega^{\rm even}(V_{\kappa,j}, \R)$ over $\kappa \in J_j$ by the definition of the trace. Using a partition of unity on $M$ we see that
\[
	\int_M \wh{A}(M) \ch_{\tau}(\mcal{E}_i) = \int_{\bar{M}} \wh{A}(\bar{M}) \ch(\psi^*F - \trivial{\C^n}) \ ,
\]
Since the class of $\ch(\psi^*F - \trivial{\C^n})$ is concentrated in degree $n$ we see that the above term is non-vanishing.
\end{proof}

\begin{remark}\label{rem:ti}
Due to the stabilization trick mentioned in the proof the fibers of $E_i$ are isomorphic to $t_iC_i$ for some projection $t_i \in C_i$, where $t_i = 1$ if $\bar{M}$ is non-compact.
\end{remark}

Having the sequence $E_i$ of almost flat twisted bundles at hand, we can form the $C^*$-algebra
\[
	A = \prod_{i \in \N} C_i 
\]
of bounded sequences with $i$th entry in $C_i$, in which the norm closure of the sequences with only finitely many non-zero entries
\[
	A' = \overline{\bigoplus_{i \in \N} C_i}^{\lVert \cdot \rVert}
\]
is a two-sided ideal and we set $Q = A / A'$. Let $A_i$ be the ideal in $A$ consisting of sequences that are $0$ everywhere, but in the $i$th entry. 

\begin{theorem} \label{thm:AssemblingEi}
There is a smooth twisted Hilbert $A$-module bundle $E \rightarrow \widetilde{M}$ together with a twisted connection 
\[
 \nabla^E : C^\infty(\wtM,E) \rightarrow C^\infty(\wtM, T^*\widetilde{M} \otimes E)
\]
such that the following holds:
\begin{itemize}
 \item $E \cdot A_i$ is isomorphic to $E_i$ as a twisted Hilbert $C_i$-module bundle.
 \item The connection preserves the subbundles $E \cdot A_i$.
 \item The sequence of curvatures $\Omega_{i} \in \Omega^2(M, \bgend{E \cdot A_i})$ of the connection induced on $E \cdot A_i$ by $\nabla^E$ satisfies \( \lim_{i \to \infty} \lVert \Omega_{i} \rVert = 0 \).
\end{itemize}
\end{theorem}

\begin{proof}\label{pf:AssemblingEi}
The idea is to see that the product bundle $E_L = \Delta_{\wtM}^*\left(\prod_{i \in \N} E_i\right)$, where 
\[
	\Delta_{\wtM} \colon \wtM \to \prod_{i \in \N} \wtM
\]
is the diagonal map, has locally Lipschitz continuous transition functions. This parallels the construction given in the proof of \cite[Lemma 2.1]{paper:HankeSchick} with the only difference that we have to work equivariantly over $\wtM$, so we just sketch the differences and refer to \cite{paper:HankeSchick} for the details: We cover $M$ by subsets $U_j$, each of them diffeomorphic to $I^n$, where $I = [0,1]$, such that $\wtM \to M$ is trivial over $U_j$ via 
\[
	\phi_j \colon U_j \times \pi \to \left.\wtM\right|_{U_j}\ .
\]
We can find trivializations 
\[
	\psi^1_{i,j} \colon \left.\phi_j^*E_i\right|_{U_j \times \{1\}} \to I^n \times t_i\,C_i
\]
of $\left.\phi_j^*E_i\right|_{U_j \times \{1\}}$, such that constant sections of $\phi_j^*E_i$ over $I^k \times \{0,\dots,0\}$ are parallel with respect to $\nabla_{\partial_l}$ for $1 \leq l \leq k$, where $\nabla$ denotes the connection induced by $\nabla^{E_i}$. Using the twisting we can extend $\psi^1_{i,j}$ to a trivialization $\psi_{i,j}$ of $\left.\phi_j^*E_i\right|_{U_j \times \pi}$ with components $\psi_{i,j}^g$ with $g \in \pi$. Let $\eta_{i,j}^g \in \Omega^1(I^n, t_i C_i t_i)$ be the pullback of the connection $1$-form of $\nabla^{E_i}$. The way the trivializations are constructed is crucial to prove the estimate given in \cite[Lemma 2.3]{paper:HankeSchick}, which now still holds and we have
\(
	\lVert \eta_{i,j}^g \rVert \leq n \cdot \lVert \Omega_{i,j}^g \rVert
\),
where $\Omega_{i,j}^g$ denotes the curvature of $\eta_{i,j}^g$. The right hand side of the above inequality is independent of $g \in \pi$. Thus, our control of the curvatures carries over to an upper bound on the local connection $1$-forms. The trivializations $\psi_{i,j}$ induce transition maps
\[
	\psi_{i,(j,k)} \colon (U_j \cap U_k) \times \pi \to U(t_i C_i t_i)
\]
and the upper bound on the local connection $1$-forms yields an upper bound on the norm of the derivative $D_{(x,g)} \psi_{i,(j,k)}$ just as described in \cite[Lemma 2.5, Proposition~2.6]{paper:HankeSchick} proving Lipschitz continuity of the transition functions. The Lipschitz condition on the twisting maps ensures that the product of the $\gamma^g_i$ is continuous, when considered as an element in $C(\wtM, \Homom{E_L}{g^*E_L})$. Thus, $E_L$ is a continuous twisted Hilbert $A$-module bundle. Note that $\pi = \pi_1(M)$ acts via the adjoint action unitarily on $C_i$ and we set $\mcal{A} = \wtM \times_{\rm Ad} C_i$. $E_L$ corresponds to a projection $t_L \in C(M, \mcal{A})$, which we can approximate by a projection in $C^\infty(M, \mcal{A})$ to obtain a smooth twisted Hilbert $A$-module bundle $E$ by the construction given in Theorem~\ref{thm:Twisted_K}. We have $E \cdot A_i \cong E_L \cdot A_i = E_i$. The isomorphism may only be continuous, but it can be smoothed. 

To construct the connection $\nabla^E$ we only need to give local connection $1$-forms over the sets $\phi_j(U_j \times \{1\}) \subset \wtM$ and extend them equivariantly via $\gamma^g$ to get connection forms over the images of $U_j \times \pi$, which can be patched together with a partition of unity on $M$. The construction takes the local forms of the $E_i$ and uses a convolution argument to get a smooth form on the product. This is exactly the same as in \cite{paper:HankeSchick}. 
\end{proof}

The twisting $\gamma^g$ maps the subbundle $E \cdot A'$ into itself, therefore the quotient $\mcal{W} = E / (E \cdot A')$ is a smooth twisted Hilbert $Q$-module bundle equipped with a \emph{flat} connection $\nabla^Q$ and typical fiber $tQ$ for some projection $t \in Q$. If we now fix a basepoint $\widetilde{m} \in \wtM$, we get a projective holonomy representation in the sense of Definition \ref{def:projhol}:
\(
	(\pi, c_{\wh{\pi}}) \to \Endo{\mcal{W}_{\widetilde{m}}} = tQt
\).
By the universal property of the maximal twisted group $C^*$-algebra, this extends to a $*$-homomorphism
\(
	\phi \colon \Cmax{\wh{\pi}}{\pi} \to Q
\).

As a consequence of Corollary \ref{cor:naturality}, the induced map $\phi_* \colon K_0(\Cmax{\wh{\pi}}{\pi}) \to K_0(Q)$ maps $\theta^{\rm max}(M)$ to ${\rm ind}(D_+^{\mcal{W}'})$, where 
\[
	\mcal{W}' = \wtM \times tQ = \wtM \times \mcal{W}_{\widetilde{m}}\ .
\]
Using parallel transport, its equivariance as described before Definition \ref{def:projhol} with respect to $\gamma^g$ and flatness of $Q$ we see that $\mcal{W}'$ is isomorphic to $\mcal{W}$ as a twisted Hilbert $Q$-module bundle. 

\begin{theorem}\label{thm:theta_max_enl}
Let $M$ be a closed compact smooth orientable even-dimensional manifold with $\dim(M) \geq 3$ and $\wtM$ spin that is enlargeable in the sense of Definition~\ref{def:enlargeable}. Then we have
\[
	\theta^{\rm max}(M) \neq 0 \in K_0(\Cmax{\wh{\pi}}{\pi})\ .
\]
\end{theorem}

\begin{proof}\label{pf:theta_max_enl}
As we saw above, we have
\(
	\phi_*(\theta^{\rm max}(M)) = {\rm ind}(D_+^{\mcal{W}}) \in K_0(Q)
\)
By \cite{paper:HankeSchickInfinite}, the group $K_0(Q)$ splits off a summand 
\[
	\prod_{i \in \N} K_0(\K) / \bigoplus_{i \in \N} K_0(\K) \cong \prod_{i \in \N} \Z / \bigoplus_{i \in \N} \Z
\]
and the image of ${\rm ind}(D_+^{\mcal{W}})$ in the latter group corresponds to the sequence
\[
	z_i = \left[p_*\left({\rm ind}(D_+^{E \cdot A_i})\right)\right] = \left[p_*\left({\rm ind}(D_+^{E_i})\right)\right]\ ,
\]
where $p \colon C_i \to \K$ is the projection. By Theorem \ref{thm:almostflat} it has only non-vanishing entries.
\end{proof}

\begin{remark}\label{rem:releasing_restrictions}
Extending the suspension argument from \cite{paper:HankeSchick} it is easy to drop the assumption about even-dimensionality. Relaxing the condition about the orientability of $M$ requires incorporating orientation twists of $K$-theory into the setup, which can be seen as a special case of twisted $\Z/2\Z$-equivariant $K$-theory as has been observed by Karoubi \cite[Remark 6.16]{paper:KaroubiOldAndNew}, \cite{paper:KaroubiClifford}. These can also be described by gerbes (see the Jandl gerbes in \cite{paper:SurfaceHolonomy} and the functor $K_{\pm}(X)$ in \cite{paper:Z2twists}, which is naturally equivalent with Karoubi's definition), therefore the above argument should generalize to non-orientable manifolds as well. Nevertheless, it seems to be impossible to drop the spin condition for the covers $\bar{M}$ in the definition of enlargeability, since our construction of a projective representation with the right cocycle relies on that.
\end{remark}

\bibliographystyle{plain} 
\bibliography{TwistedAndPSC}

\end{document}